\documentclass[12pt,a4paper]{article}
 
\usepackage{amsmath}
\usepackage{amssymb}
\usepackage{amsthm}
\usepackage{float}
\usepackage{graphicx}
\usepackage{hyperref,url}
\usepackage[boxsize=0.7em]{ytableau}
\usepackage{array}
\hypersetup{
	colorlinks=true,
  linkcolor=red,          
  citecolor=blue,         
  filecolor=magenta,      
  urlcolor=cyan           
}

\newtheorem{prop}{Proposition}[section]
\newtheorem{lemma}[prop]{Lemma}
\newtheorem{theorem}[prop]{Theorem}

\newtheorem{corollary}[prop]{Corollary}
\theoremstyle{definition}
\newtheorem{definition}[prop]{Definition}
\newtheorem{remark}[prop]{Remark}
\newtheorem{example}[prop]{Example}

\newcommand{\N}{\mathbb{N}}
\newcommand{\R}{\mathbb{R}}

\newcommand{\PP}{\mathcal{P}}

\newcommand{\PART}{\textnormal{Partition}}
\newcommand{\RANK}{\textnormal{rank}}
\newcommand{\seqnum}[1]{\href{https://oeis.org/#1}{\rm \underline{#1}}}
\begin{document}

\title{The minimal sum of squares over partitions with a nonnegative rank}

\author{Sela Fried\thanks{A postdoctoral fellow in the Department of Computer Science at the Ben-Gurion University of the Negev, Israel.} 
\\ \href{mailto:friedsela@gmail.com}{friedsela@gmail.com}}
\date{} 
\maketitle

\begin{abstract}
Motivated by a question of Defant and Propp (2020) regarding the connection between the degrees of noninvertibility of functions and those of their iterates, we address the combinatorial optimization problem of minimizing the sum of squares over partitions of $n$ with a nonnegative rank. Denoting the sequence of the minima by $(m_n)_{n\in\mathbb{N}}$, we prove that $m_n=\Theta\left(n^{4/3}\right)$. Consequently, we improve by a factor of $2$ the lower bound provided by Defant and Propp for iterates of order two.
\end{abstract} 

\section{Introduction}

Recently, \cite{DP} defined the \emph{degree of noninvertibility} of a function $f\colon X\to Y$ between two finite nonempty sets $X$ and $Y$ by $$\deg(f)=\frac{1}{|X|}\sum_{x\in X}\left|f^{-1}(f(x))\right|,$$ as a measure of how far $f$ is from being injective. Interested mainly in endofunctions (also called dynamical systems within the field of dynamical algebraic combinatorics), that is, functions $f\colon X\to X$, they then computed the degrees of noninvertibility of several specific functions and studied, from an extremal point of view, the connection between the degrees of noninvertibility of functions and those of their iterates. They concluded their work with the following question: Let $2\leq k\in\N$. Does the limit \begin{equation}\label{eq;44}
\lim_{n\to\infty}\max_{\substack{f\colon X\to X\\ |X|=n}}\frac{\deg\left(f^k\right)}{\deg(f)^{2-1/2^{k-1}}}\frac{1}{n^{1-1/2^{k-1}}}\end{equation} exist? If so, what is its value? They remarked that even answering the question for $k=2$ would be interesting and stated that it follows from their results that, if the limit exists, then it lies in the interval between $3^{-3/2}\approx0.19245$ and $1$. 

Our attempts to answer their question in the case $k=2$ have led us to a combinatorial optimization problem that seems not to have been addressed before, namely, the problem of finding the minimal sum of squares over partitions with a nonnegative rank. In this work, we address this problem and, consequently, improve the lower bound of the interval $\left[3^{-3/2}, 1\right]$ by a factor of $2$. 

We begin by stating our main results. The definitions of the terms that we use and the proofs of all the statements are given in Section \ref{43}. 

\section{Main results}

Let $X$ be a set of size $n\in\N$ to be used throughout this work. We denote by $\N_0$ the set of all nonnegative integers. Taking $k=2$ in (\ref{eq;44}), we wish to lower bound 
\begin{equation}\label{eq;001} 
\max_{f\colon X\to X}\frac{\deg\left(f^2\right)}{\deg(f)^{3/2}}\frac{1}{n^{1/2}},
\end{equation} where  $f^2$ stands for the composition $f\circ f$. Our approach is based on the fact that the functions with the largest possible degree of noninvertibility, namely $n$, are the constant functions (cf.\ \cite[p.\ 2]{DP}). Thus, we wish to solve the following combinatorial optimization problem:

\begin{align}
\text{minimize } & \deg(f)\label{eq;6522}\\
\text{where } & f\colon X\to X  \textnormal{ is such that } f^2 \textnormal{ is constant}.\nonumber
\end{align} 

The notion of the degree of noninvertibility of a function $f\colon X\to X$ is directly related to the sum of squares over a certain partition of $n$ via the observation that
$$
\deg(f)=\frac{1}{n}\sum_{x\in X}\left|f^{-1}(x)\right|^2
$$ (cf.\ \cite[p.\ 2]{DP}). Indeed, if $X=\{1,\ldots,n\}$ then $\left|f^{-1}(1)\right|,\ldots,\left|f^{-1}(n)\right|$ yield, upon reordering and omitting zeros, a partition of $n$ that we denote by $\PART(f)$. Conversely, it is clear that every partition $\lambda$ of $n$ induces a function $f\colon X\to X$ such that $\PART(f)=\lambda$.

It turns out (cf.\ Lemma \ref{lem;732}), that if $f\colon X\to X$ is such that $f^2$ is constant, then $\PART(f)$ has a nonnegative rank (cf.\ Definition \ref{d;01}). 
Denoting the set of all partitions of $n$ by $\PP(n)$ and the Euclidean norm of a vector $x$ by $||x||_2$, we may rewrite problem (\ref{eq;6522}) equivalently as

\begin{align}
    \text{minimize } & ||\lambda||_2^2 \label{eq;364}\\
    \text{where } & \lambda\in\PP(n) \textnormal{ is such that } \RANK(\lambda)\geq 0. \nonumber 
\end{align} 

\begin{remark}
Notice that, in general, a minimizer of (\ref{eq;364}) is not unique. For example, both $(5,3, 3, 3, 3)$ and $(6,3, 2, 2, 2, 2)$ minimize (\ref{eq;364}) for $n=17$.
\end{remark}

Our first main result is the observation that, for $n\neq 2$, the partitions of $n$ that minimize (\ref{eq;364}) must have a certain structure, namely, their largest part $\lambda_1$ is equal to their number of parts (i.e., their rank is $0$) and $n-\lambda_1$ is divided as evenly as possible among the remaining $\lambda_1 -1$ parts (see Figure \ref{fig:M1} for a visualization):

\begin{theorem}\label{thm;2}
For $n\geq 2$, problem (\ref{eq;364}) is equivalent to the following problem:
 \begin{align}
\textnormal{minimize } & x^{2}+r(a+1)^{2}+(x-1-r)a^{2}\label{eq;65}\\
\textnormal{such that } & x\in\{2,\ldots, n\}, \nonumber\\& n-x = a(x-1) + r \textnormal{ where } a\in\N_0 \textnormal{ and } 0\leq r<x-1\textnormal{ and}\nonumber\\ &x\geq \begin{cases}a,&r = 0;\\a+1,&\textnormal{otherwise}.\end{cases}\nonumber
\end{align} 
\end{theorem}

\begin{figure}
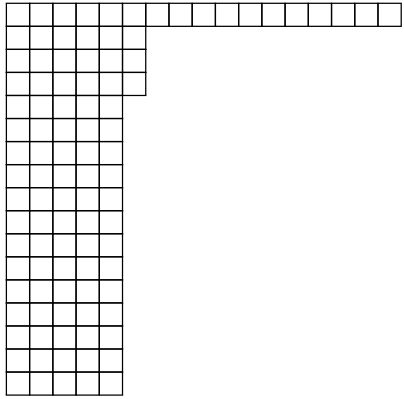

\centering
\ydiagram{17, 6, 6, 6, 5, 5, 5, 5, 5, 5, 5, 5, 5, 5, 5, 5, 5}
\caption{The Young diagram of one of the two minimizers of (\ref{eq;364}), for $n=100$.}\label{fig:M1}
\end{figure}

Let $m_n$ denote the minimum of (\ref{eq;65}). Then $$(m_n)_{n\in\N}=1, 4, 5, 8, 11, 14, 17, 22, 25,\;\ldots$$ (the sequence is registered as \seqnum{A353044} in the OEIS. See Table \ref{table:22} for its first $210$ values). Lemmas \ref{lem;660} and \ref{lem;661}, respectively, show that $(m_n)_{n\in\N}$ is strictly increasing and that its elements have alternating parity. While we do not have an exact formula for $(m_n)_{n\in\N}$, we obtain lower and upper bounds by applying continuous relaxation:

\begin{theorem}\label{th;211}
We have $m_n=\Theta(n^{4/3})$. More precisely, $$\frac{n^{4/3}}{4}\leq m_n\leq(2^{-2/3}+2^{1/3})n^{4/3},$$ for $n\geq 28$.
\end{theorem}

Theorem \ref{th;211} allows us to improve by a factor of $2$ the lower bound given by \cite[p.\ 17]{DP}:

\begin{corollary}\label{cor;222}
Taking $k=2$, the limit in (\ref{eq;44}), if it exists, is lower bounded by $2\cdot3^{-3/2}$.
\end{corollary}

We proceed by showing that, if $t_n$ is a minimizer of (\ref{eq;65}), then it must lie in a certain interval:

\begin{theorem}\label{T;6}
Suppose that $n\geq 6$. There exists a function $u_n\colon(1,\infty)\to\R$ such that if $x_0^{(n)}$ is the global minimum of $u_n$ and $x_1^{(n)} < x_2^{(n)}$ are the two real positive roots of the cubic polynomial $$
x^{3}+\left(-2n-u_n\left(\left\lfloor x_0^{(n)}\right\rfloor\right)\right)x+n^{2}+u_n\left(\left\lfloor x_0^{(n)}\right\rfloor\right),$$ then $t_n\in \left\{\left\lceil x_1^{(n)} \right\rceil, \ldots, \left\lfloor x_2^{(n)} \right\rfloor\right\}$, for every $t_n$ that minimizes (\ref{eq;65}). 
\end{theorem}

\begin{example}
In the notation of Theorem \ref{T;6}, we have $\left\lceil x_1^{(1000)}\right\rceil = 78, \left\lfloor x_2^{(1000)}\right\rfloor = 82$ and $t_{1000} = 78$ is the unique minimizer of (\ref{eq;65}). See Figure \ref{fig;eg} for a  visualization of Theorem \ref{T;6}.
\end{example}

\begin{figure}[H]
\includegraphics[width=10cm, height=8cm]{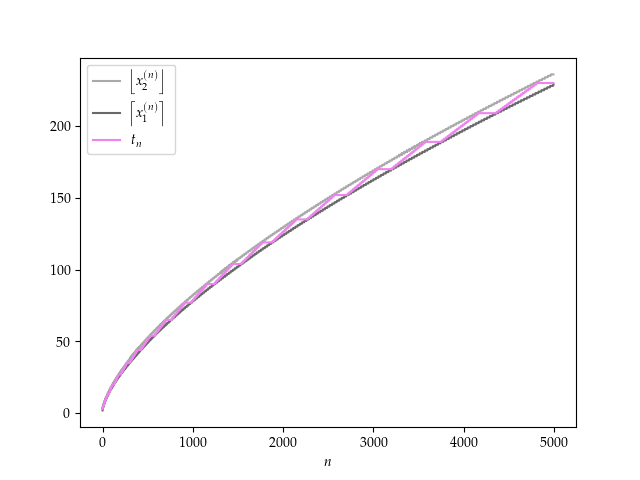}
\centering
\caption{The bounds on $t_n$ of Theorem \ref{T;6}, together with the smallest $t_n$ that minimizes (\ref{eq;65}), for $6\leq n\leq 5000$.}
\label{fig;eg}
\end{figure}

Finally, we establish the asymptotic behaviour of the minimizers of (\ref{eq;65}): 
\begin{theorem}\label{thm;456}
Let $(t_n)_{n\in\N}$ be any sequence such that $t_n$ minimizes (\ref{eq;65}). Then $t_n=\Theta\left(n^{2/3}\right)$.
\end{theorem}

\section{Definitions and proofs}\label{43}

We begin with the definition of the rank of a partition, a notion that goes back to \cite{D} (see also \seqnum{A064174} in the OEIS). The reader is referred to \cite{A} for the general theory of partitions.

\begin{definition}\label{d;01}
Let $\lambda\in\PP(n)$. The \emph{rank of $\lambda$}, denoted by $\RANK(\lambda)$, is defined as $\lambda$'s largest part minus its number of parts.
\end{definition}

Functions $f\colon X\to X$ such that $f^2$ is constant are characterized by induced partitions of $n$ with a nonnegative rank:

\begin{lemma}\label{lem;732}
Let $f\colon X\to X$ be a function such that $f^2$ is constant. Then $\RANK(\PART(f))\geq 0$. Conversely, for every $\lambda\in\PP(n)$ with $\RANK(\lambda)\geq 0$, there is a function $f\colon X\to X$ such that $f^2$ is constant and $\PART(f)=\lambda$. 
\end{lemma}

\begin{proof}
Assume that $f^2$ is constant. Then there is $y\in X$ such that $f(f(x))=y$ for every $x\in X$. Thus, $f(x) \in f^{-1}(y)$ for every $x\in X$. Notice that the number of parts of $\PART(f)$ is equal to $\left|\text{Im}(f)\right|$. Now, $$|\text{Im}(f)|\leq \left|f^{-1}(y)\right|\leq \max_{x\in X}\left\{\left|f^{-1}(x)\right|\right\}.$$ It follows that $\RANK(\PART(f))\geq 0$.

Conversely, suppose $X=\{1,\ldots,n\}$ and let $\lambda=(\lambda_1,\ldots,\lambda_r)\in\PP(n)$ such that $\RANK(\lambda)\geq 0$. We define $f\colon X\to X$ as follows: For $1\leq i\leq n$ let $$f(i) = \begin{cases} 1, &  1\leq i\leq \lambda_1; \\ k, & \sum_{j=1}^{k-1} \lambda_j <i\leq \sum_{j=1}^k\lambda_j \textnormal{ where } 2\leq k\leq r.\end{cases}$$ Since $\lambda_1\geq r$, we have that $2,\ldots, r\in f^{-1}(1)$. It follows that $f(f(i))=1$ for every $i\in X$, i.e., $f^2$ is constant. Furthermore, $\PART(f)=\lambda$.
\end{proof}

The proof of Theorem \ref{thm;2} relies on the following two lemmas. For the first, we shall need the notion of a balanced partition. We could not find any mention of this notion other than in \seqnum{A047993} in the OEIS. 

\begin{definition}
A partition whose rank is zero is called a \emph{balanced} partition.
\end{definition}

\begin{lemma}\label{lem;15}
If $n\neq 2$, then the minimum of (\ref{eq;65}) is obtained at a balanced partition.
\end{lemma}

\begin{proof}
Consider first the cases $n=1,3,4$: If $n=1$, then there is only one partition $(1)$ which is balanced. If $n=3$, then there are only two partitions with a nonnegative rank, namely $(3)$ and $(2,1)$, of which the latter, that is balanced, has the smallest sum of squares. Similarly, if $n=4$, then there are only three partitions with a nonnegative rank, namely $(4), (3,1)$ and $(2,2)$, of which the latter, that is balanced, has the smallest sum of squares.

Assume now that $n\geq 5$ and let $\lambda=(\lambda_1,\ldots,\lambda_r)\in\PP(n)$ such that $\RANK(\lambda)>0$. We shall construct a partition $\lambda'\in\PP(n)$ such that $||\lambda||_2^2> ||\lambda'||_2^2$ and $0 \leq \RANK(\lambda')<\RANK(\lambda)$. To this end, let $k=\max\{1\leq i\leq r\;|\; \lambda_i>1\}$. We distinguish between two cases:
\begin{enumerate}
\item $k>1$. We have
\begin{align}
||\lambda||_2^2&=\sum_{1\leq i\leq r, i\neq k}\lambda_{i}^{2}+(\lambda_k-1+1)^2\nonumber\\&=\sum_{1\leq i\leq r, i\neq k}\lambda_{i}^{2}+(\lambda_k-1)^2+\overbrace{2\lambda_k-1}^{>1}\nonumber\\&>\sum_{1\leq i\leq r, i\neq k}\lambda_{i}^{2}+(\lambda_k-1)^2+1\nonumber\\&= ||\lambda'||_2^2,\nonumber
\end{align} where $\lambda'=(\lambda_1,\ldots,\lambda_{k-1},\lambda_k-1,\lambda_{k+1},\ldots,\lambda_r,1)$. Since $\lambda_1\geq r+1$, we have  $\RANK(\lambda')\geq 0$. Furthermore,  $\RANK(\lambda')<\RANK(\lambda)$.
\item $k=1$. In this case, $\lambda = (n-r+1,\overbrace{1,\ldots,1}^{r-1 \textnormal{ times }})$. First, assume that $r\geq 3$. It is easy to see that $$(n-r+1)^{2}+r-1>(n-r)^{2}+4+r-2\iff n-r> 1.$$ Now, by assumption, $n-r+1>r$. Thus, $n-r>r-1\geq 2$ and we take $\lambda'=(n-r,2,\overbrace{1,\ldots,1}^{r-2 \textnormal{ times }})$. 

Assume now that $r=2$. Then $\lambda=(n-1,1)$ and it is easy to see that $$(n-1)^{2}+1>(n-2)^{2}+2\iff n\geq3.$$ Thus, we take $\lambda'=(n-2,1,1)$.

Finally, assume that $r=1$. Then $\lambda=(n)$ and we have $$n^2 > (n-1)^2 +1 \iff n\geq 2.$$ Then we take $\lambda'=(n-1,1)$. 

In each of these cases, $||\lambda||_2^2> ||\lambda'||_2^2$ and $0\leq \RANK(\lambda')<\RANK(\lambda)$.
\end{enumerate}

\end{proof}

\begin{lemma}\label{lem;16}
Let $\lambda=(\lambda_1,\ldots,\lambda_r)\in\PP(n)$ such that $\lambda_j > \lambda_k + 1$ for some $1\leq j<k\leq r$. Let $\lambda'\in\PP(n)$ correspond to the parts $\lambda_1,\ldots,\lambda_j-1,\ldots,\lambda_k+1,\ldots,\lambda_r$. Then $||\lambda||_2^2>||\lambda'||_2^2$.
\end{lemma}

\begin{proof}
It suffices to prove that  $$\lambda_j^{2}+\lambda_k^{2}>(\lambda_{j}-1)^{2}+(\lambda_{k}+1)^{2},$$ which is easily seen to be equivalent to $\lambda_j>\lambda_k+1$.
\end{proof}

\paragraph{Proof of Theorem \ref{thm;2}}
The assertion follows immediately from the combination of Lemma \ref{lem;15} together with Lemma \ref{lem;16}.
\qed

In our work we shall make extensive use of two functions $l_n,u_n\colon \R\setminus\{1\}\to\R$, given by \begin{align}
l_n(x) &=x^2+\frac{(n-x)^2}{x-1} \textnormal{ and}\nonumber\\
u_n(x)&=x^2+\frac{(n-x)^2}{x-1}+\frac{x-1}{4}. \nonumber
\end{align} 

The bounds in the following lemma are visualized in Figure \ref{fig;e} for $n=100$.

\begin{lemma}\label{lem;52}
Let $2\leq x,n\in\N$. Then
\begin{equation}\label{eq;dfa}
l_n(x)\leq x^{2}+r(a+1)^{2}+(x-1-r)a^{2} \leq u_n(x),
\end{equation} where $a\in\N_0$ and $0\leq r<x-1$ are such that $n-x = a(x-1) + r$.
\end{lemma}

\begin{proof}
Let $m\in\N_0, q\in\N$ and write $m = bq + s$ where $b\in\N_0$ and $0\leq s<q$. It is immediately verified that $$s(b+1)^{2}+(q-s)b^{2}-\frac{m^2}{q}=\frac{s(q-s)}{q}.$$ Clearly, $\frac{s(q-s)}{q}\geq 0$. On the other hand, $\frac{s(q-s)}{q}\leq \frac{q}{4}$ due to the well known fact that the maximal product of two integers whose sum is $q$ is $\left\lfloor \frac{q^2}{4}\right\rfloor\leq \frac{q^2}{4}$ (cf.\ \seqnum{A002620} in the OEIS). It follows that \begin{equation}\label{eq;dfb}\frac{m^2}{q}\leq s(b+1)^{2}+ (q-s)b^{2}\leq \frac{m^2}{q}+\frac{q}{4}.\end{equation} 
The two inequalities in (\ref{eq;dfa}) are now proved by setting $m=n-x, q=x-1$ and adding $x^2$ to (\ref{eq;dfb}).
\end{proof}

\begin{figure}
\includegraphics[width=10cm, height=8cm]{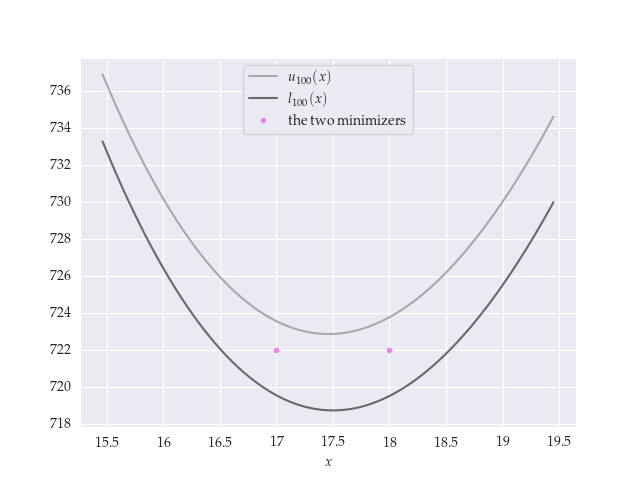}
\centering
\caption{The lower and upper bounds of Lemma \ref{lem;52}, visualized for $n=100$.}
\label{fig;e}
\end{figure}

We may now prove Theorem \ref{th;211}. 

\paragraph{Proof of Theorem \ref{th;211}}
The function $u_n(x)$ is continuous in $(1,\infty)$ and $$\lim_{x\to 1^+}u_n(x) = \lim_{x\to \infty}u_n(x)=\infty.$$ Furthermore,
$$u'_n(x)=\frac{8x^{3}-11x^{2}-2x+8n-4n^{2}+1}{4(x-1)^{2}}.$$ Since the discriminant of the numerator of $u'_n(x)$ is negative for $n\geq 3$, the equation $u'_n(x)=0$ has a unique real solution $x_0^{(n)}$, given by $x_0^{(n)}=\frac{11+C_n+169/C_n}{24}$, where $$C_n=\sqrt[3]{3456n^{2}-6912n+1259-\sqrt{(3456n^{2}-6912n+1259)^{2}-169^3}}.$$ It follows that, restricted to $(1,\infty)$, the function $u_n(x)$ obtains its global minimum at $x_0^{(n)}$. Now, for every $0<y<z\in\R$, we have $$\frac{y^{2}}{2z}\leq z-\sqrt{z^{2}-y^{2}}\leq\frac{y^{2}}{z}.$$ Thus, $C_n\leq 1$ for $n\geq 28$ and therefore \begin{align}
x_0^{(n)}&=\frac{11+C_n+169/C_n}{24}\nonumber\\&\leq\frac{1}{2}+\frac{169}{24} \sqrt[3]{\frac{2(3456n^{2}-6912n+1259)}{169^3}}\nonumber\\&\leq\frac{1}{2}+2^{-1/3}n^{2/3}\nonumber
\end{align} (notice, for later use, that $\lim_{n\to\infty}\frac{x_0^{(n)}}{n^{2/3}}=2^{-1/3}$).
Since $u_n(x)$ is increasing in $\left[x_0^{(n)}, \infty\right)$, we have 
\begin{align}
u_n\left(\left\lceil x_0^{(n)} \right\rceil\right)&\leq u_n\left(x_0^{(n)}+1\right)\nonumber\\&\leq  u_n\left(\frac{3}{2}+2^{-1/3}n^{2/3}\right)\nonumber\\ &=\left(\frac{3}{2}+2^{-1/3}n^{2/3}\right)^2+\frac{\left(n-\left(\frac{3}{2}+2^{-1/3}n^{2/3}\right)\right)^2}{\frac{1}{2}+2^{-1/3}n^{2/3}}+\frac{\frac{1}{2}+2^{-1/3}n^{2/3}}{4}\nonumber\\&\leq \left(2^{-2/3}+2^{1/3}\right)n^{4/3},\nonumber
\end{align} where the last inequality holds for $n\geq 5$. Now, it follows from Lemma \ref{lem;52} that $m_n\leq u_n\left(\left\lceil x_0^{(n)} \right\rceil\right)$, which concludes the proof of the upper bound. 

To prove the lower bound, we notice that, for $n\geq 3$ and restricted to $(1,\infty)$, the function $l_n(x)$ obtains its global minimum at $y_0^{(n)}$, given by $y_0^{(n)}=\frac{D_n+1+1/D_n}{2}$, where $$D_n=\sqrt[3]{(2n^{2}-4n+1)-\sqrt{(2n^{2}-4n+1)^{2}-1}}.$$ We have 
\begin{align}
    l_n\left(y^{(n)}_0\right)&=\left(y^{(n)}_0\right)^{2}+\frac{\left(y_0^{(n)}-n\right)^{2}}{y_0^{(n)}-1}\nonumber\\&\geq\frac{1}{4D_n^{2}}\nonumber\\&\geq\frac{\sqrt[3]{(2n^{2}-4n+1)^{2}}}{4}\nonumber\\&\geq\frac{n^{4/3}}{4},\nonumber
\end{align} where the last inequality holds for $n\geq 4$. By Lemma \ref{lem;52}, $m_n\geq l_n\left(y_0^{(n)}\right)$.
\qed

\paragraph{Proof of Corollary \ref{cor;222}}
By Theorem \ref{th;211}, if $n\geq 28$, then $m_n \leq(2^{-2/3}+2^{1/3})n^{4/3}$. Let $f\colon X\to X$ be a function such that $||\PART(f)||^2_2=m_n$. Notice that $\deg(f)=\frac{m_n}{n}$. It follows that, if the limit in (\ref{eq;44}) exists for $k=2$, then
\begin{align}
\lim_{n\to\infty}\max_{f\colon X\to X}\frac{\deg(f^2)}{\deg(f)^{3/2}}\frac{1}{n^{1/2}}&\geq\lim_{n\to\infty} \frac{n}{\left(\left(2^{-2/3}+2^{1/3}\right)n^{1/3}\right)^{3/2}}\frac{1}{n^{1/2}}\nonumber\\&=2\cdot3^{-3/2}.\nonumber\end{align}
\qed

\paragraph{Proof of Theorem \ref{T;6}}
Let $x_0^{(n)}$ and $y_0^{(n)}$ be the points, calculated in the proof of Theorem \ref{th;211}, at which $u_n$ and $l_n$, respectively, obtain their global minimum. We wish to solve the equation 
\begin{equation}\label{eq;981}
l_n(x) = u_n\left(\left\lfloor x_0^{(n)}\right\rfloor\right)
\end{equation} (see Figure \ref{fig;egb} for a visualization for $n=100$). The function $l_n(x)$ is continuous in $(1,\infty)$ and $$\lim_{x\to 1^+}l_n(x) = \lim_{x\to \infty}l_n(x)=\infty.$$ Since $l_n\left(y_0^{(n)}\right)< u_n\left(x_0^{(n)}\right)\leq u_n\left(\left\lfloor x_0^{(n)}\right\rfloor\right)$, by the mean value theorem, equation (\ref{eq;981}) has at least two real solutions in $(1,\infty)$. Similarly, $l_n(x)$ is continuous in $(-\infty, 1)$ and $$\lim_{x\to 1^-}l_n(x) = -\infty, \lim_{x\to -\infty}l_n(x)=\infty.$$ Thus, equation (\ref{eq;981}) has at least one real solution in $(-\infty, 1)$ and, since solving it is equivalent to finding the roots of the cubic polynomial $$x^{3}+\left(-2n-u_n\left(\left\lfloor x_0^{(n)}\right\rfloor\right)\right)x+n^{2}+u_n\left(\left\lfloor x_0^{(n)}\right\rfloor\right),$$ we conclude that equation (\ref{eq;981}) has exactly two real solutions 
$1<x_1^{(n)}<x_2^{(n)}$. Necessarily, $x_1^{(n)}< \left\lfloor x_0^{(n)}\right\rfloor, y_0^{(n)}< x_2^{(n)}$. Thus, if $t_n$ minimizes (\ref{eq;65}), then $t_n\in \left\{\left\lceil x_1^{(n)} \right\rceil, \ldots,  \left\lfloor x_2^{(n)} \right\rfloor\right\}$.
\qed

\begin{figure}[H]
\includegraphics[width=10cm, height=8cm]{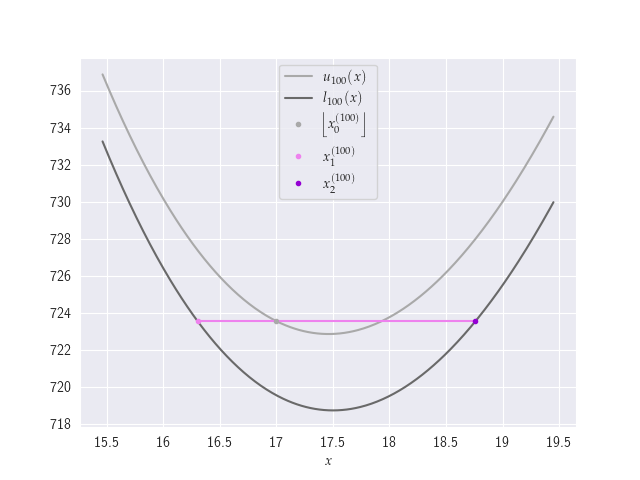}
\centering
\caption{We have $\left\lceil x_1^{(100)}\right\rceil=17$ and $\left\lfloor x_2^{(100)}\right\rfloor=18$. Thus, $t_{100}\in\{17,18\}$.}
\label{fig;egb}
\end{figure}

In the proof of Theorem \ref{thm;456} we shall make use of the following notation (cf. \cite[(9.3)]{G}): Let $(a_n)_{n\in\N}$ and $(b_n)_{n\in\N}$ be two sequences. By $a_n \prec b_n$ we mean that $\lim_{n\to\infty}\frac{a_n}{b_n}=0$.

\paragraph{Proof of Theorem \ref{thm;456}}
Let $x_0^{(n)}, x_1^{(n)}$ and $x_2^{(n)}$ be as in Theorem \ref{T;6} and consider the cubic polynomial \begin{equation}\label{eq;hse}
x^{3}+\left(-2n-u_n\left(\left\lfloor x_0^{(n)}\right\rfloor\right)\right)x+n^{2}+u_n\left(\left\lfloor x_0^{(n)}\right\rfloor\right).\end{equation}
Since $x_0^{(n)},t_n\in \left\{\left\lceil x_1^{(n)} \right\rceil, \ldots, \left\lfloor x_2^{(n)} \right\rfloor\right\}$ and $x_0^{(n)}=\Theta\left(n^{2/3}\right)$, it suffices to show that $x_2^{(n)}-x_1^{(n)}\prec n^{2/3}$. To this end, denote $p_n=-2n-u_n\left(\left\lfloor x_0^{(n)}\right\rfloor\right)$ and $q_n=n^{2}+u_n\left(\left\lfloor x_0^{(n)}\right\rfloor\right)$. By Cardano's formula (e.g., \cite[p. 128]{V}), the three roots of (\ref{eq;hse}) are given by $$\sqrt[3]{-\frac{q_n}{2}+\sqrt{\frac{p_n^3}{27}+\frac{q_n^2}{4}}}+\sqrt[3]{-\frac{q_n}{2}-\sqrt{\frac{p_n^3}{27}+\frac{q_n^2}{4}}}.$$ In the appendix we show that $\frac{p_n^3}{27}+\frac{q_n^2}{4}\prec n^4$. Since $u_n\left(\left\lfloor x_0^{(n)}\right\rfloor\right)=\Theta\left(n^{4/3}\right)$, we have $q_n=\Theta\left(n^2\right)$. Hence, $$-\frac{q_n}{2}+\sqrt{\frac{p_n^3}{27}+\frac{q_n^2}{4}}= r_n(\cos(\pi-\theta_n) + i\sin(\pi-\theta_n)),$$ where $r_n=\Theta\left(n^2\right)$ and $\theta_n = \arctan\left(\frac{2\left|\sqrt{\frac{p_n^3}{27}+\frac{q_n^2}{4}}\right|}{q_n}\right)$. Notice that $\lim_{n\to\infty}\theta_n=0$. Proceeding as in \cite[Example 3.106]{V}, we conclude that $$
    x_1^{(n)}  =r_n^{1/3}\cos\left(\frac{\pi + \theta_n}{3}\right) \;\textnormal{ and }\; x_2^{(n)}=r_n^{1/3}\cos\left(\frac{\pi -\theta_n}{3}\right).$$
Thus, applying the trigonometric identity
$$\cos\alpha-\cos\beta = 2\sin\left(\frac{\beta+\alpha}{2}\right)\sin\left(\frac{\beta-\alpha}{2}\right),$$ that holds for every $\alpha,\beta\in\R$, we see that

$$\lim_{n\to\infty}\frac{x_2^{(n)} - x_1^{(n)}}{n^{2/3}} =\lim_{n\to\infty}\left(\frac{r_n}{n^2}\right)^{1/3}\sin\left(\frac{\pi}{3}\right)\sin\left(\frac{\theta_n}{3}\right)=0.$$
\qed

\begin{remark}
It should be emphasized, that our approach does not, in general, provide the true maximum of (\ref{eq;001}). For example, let $n=8$ and assume that $X=\{1,\ldots,8\}$. Consider the function $h\colon X\to X$ given by $$h(i) = \begin{cases} 1, & i\in\{1,2,3\};\\ 2, & i\in\{4,5\};\\ 3, & i\in\{6,7\};\\ 4, & i=8.
\end{cases}$$ Then, $h^2\colon X\to X$ is given by $$h^2(i) = \begin{cases} 1, & i\in\{1,\ldots,7\};\\ 2, & i=8.
\end{cases}$$ It follows that $$\frac{\deg\left(h^2\right)}{\deg(h)^{3/2}} = \frac{\frac{50}{8}}{\left(\frac{18}{8}\right)^{3/2}}=\frac{50}{27}\approx 1.85185.$$ In contrast, our approach provides the partition $(3,3,2)$ that corresponds to a function $f\colon X\to X$ such that $$\frac{\deg\left(f^2\right)}{\deg(f)^{3/2}}=\frac{8}{\left(\frac{22}{8}\right)^{3/2}}\approx 1.75424.$$
\end{remark}

\begin{lemma}\label{lem;660}
The sequence $(m_n)_{n\in\N}$ is strictly increasing.
\end{lemma}
\begin{proof}
Assume that $m_{n+1}\leq m_n$ for some $n\in\N$ and let $\lambda=(\lambda_{1},\ldots,\lambda_{r})\in\PP(n+1)$ such that $||\lambda||_2^2=m_{n+1}$. Then $\lambda'=(\lambda_{1},\ldots,\lambda_{r-1}, \lambda_{r}-1)\in\PP(n)$ (omitting the last part, if necessary) such that $\RANK(\lambda')\geq 0$. Now, $$||\lambda'||_{2}^{2}<||\lambda||_{2}^{2}=m_{n+1}\leq m_n,$$ contradicting the minimality of $m_n$.
\end{proof}

\begin{lemma}\label{lem;661}
Let $\lambda=(\lambda_{1},\ldots,\lambda_{r})\in\PP(n)$. Then $n$ and $||\lambda||_2^2$ have the same parity. In particular, $n$ and $m_n$ have the same parity.
\end{lemma}
\begin{proof}
We proceed by induction on $n$. For $n=1$, the assertion holds trivially. Assume that the assertion holds for $n\in\N$ and let $\lambda=(\lambda_{1},\ldots,\lambda_{r})\in\PP(n+1)$. Then $\lambda'=(\lambda_{1},\ldots,\lambda_{r-1}, \lambda_{r}-1)\in\PP(n)$ (omitting the last part, if necessary). Now, \begin{align}
||\lambda||_2^2&=\lambda_1^2+\cdots+\lambda_r^2\nonumber\\&=\lambda_1^2+\cdots+\lambda_{r-1}^2+(\lambda_r-1)^2 +2\lambda_r -1\nonumber\\&=\overbrace{||\lambda'||_2^2}^{=\textnormal{ parity of }n-1}+\overbrace{2\lambda_r -1}^{\textnormal{odd}}\nonumber\\&=\textnormal{ parity of }n.\nonumber
\end{align}
\end{proof}

\section{Appendix}

Denote $z_n=\left\lfloor x_0^{(n)}\right\rfloor$ and $X_i=\frac{1}{(z_n-1)^i}$ for $i=1,2,3$. We have

\begin{align}
\frac{p_n^3}{27}+\frac{q_n^2}{4}&=\boxed{-\frac{n^{6}}{27}X_{3}}+\frac{2n^{5}z_{n}}{9}X_{3}-\frac{2n^{5}}{9}X_{2}-\frac{5n^{4}z_{n}^{2}}{9}X_{3}\boxed{-\frac{n^{4}z_{n}^{2}}{9}X_{2}}+\frac{31n^{4}z_{n}}{36}X^{2}\boxed{+\frac{n^{4}}{4}}\nonumber	
\\&+\frac{5n^{4}}{18}X_{2}+\frac{n^{4}}{18}X_{1}+\frac{20n^{3}z_{n}^{3}}{27}X_{3}+\frac{4n^{3}z_{n}^{3}}{9}X_{2}-\frac{11n^{3}z_{n}^{2}}{9}X_{2}-\frac{4n^{3}z_{n}^{2}}{9}X_{1}\nonumber	
\\&-\frac{10n^{3}z_{n}}{9}X_{2}-\frac{2n^{3}z_{n}}{9}X_{1}	-\frac{8n^{3}}{27}+\frac{n^{3}}{9}X_{1}-\frac{5n^{2}z_{n}^{4}}{9}X_{3}-\frac{2n^{2}z_{n}^{4}}{3}X_{2}\boxed{-\frac{n^{2}z_{n}^{4}}{9}X_{1}}	\nonumber	
\\&+\frac{13n^{2}z_{n}^{3}}{18}X_{2}+\frac{5n^{2}z_{n}^{3}}{6}X_{1}+\frac{n^{2}z_{n}^{2}}{18}+\frac{5n^{2}z_{n}^{2}}{3}X_{2}+\frac{119n^{2}z_{n}^{2}}{144}X_{1}+\frac{n^{2}z_{n}}{72}-\frac{n^{2}z_{n}}{12}X_{1}\nonumber	
\\&	-\frac{n^{2}}{72}-\frac{19n^{2}}{144}X_{1}+\frac{2nz_{n}^{5}}{9}X_{3}+\frac{4nz_{n}^{5}}{9}X_{2}+\frac{2nz_{n}^{5}}{9}X_{1}-\frac{2nz_{n}^{4}}{9}-\frac{nz_{n}^{4}}{9}X_{2}-\frac{nz_{n}^{4}}{3}X_{1}	\nonumber	
\\&-\frac{nz_{n}^{3}}{9}-\frac{10nz_{n}^{3}}{9}X_{2}-\frac{29nz_{n}^{3}}{24}X_{1}+\frac{7nz_{n}^{2}}{72}-\frac{nz_{n}^{2}}{6}X_{1}+\frac{nz_{n}}{36}+\frac{19nz_{n}}{72}X_{1}-\frac{n}{72}\nonumber	
\\&\boxed{-\frac{z_{n}^{6}}{27}}-\frac{z_{n}^{6}}{27}X_{3}-\frac{z_{n}^{6}}{9}X_{2}-\frac{z_{n}^{6}}{9}X_{1}-\frac{z_{n}^{5}}{36}-\frac{z_{n}^{5}}{36}X_{2}-\frac{z_{n}^{5}}{18}X_{1}+\frac{13z_{n}^{4}}{48}+\frac{5z_{n}^{4}}{18}X_{2}\nonumber\\ &+\frac{79z_{n}^{4}}{144}X_{1}+\frac{239z_{n}^{3}}{1728}+\frac{5z_{n}^{3}}{36}X_{1}-\frac{11z_{n}^{2}}{96}-\frac{19z_{n}^{2}}{144}X_{1}-\frac{19z_{n}}{576}+\frac{7}{432}.\nonumber
\end{align} Recall (cf.\ the proof of Theorem \ref{th;211}) that $\lim_{n\to\infty} \frac{z_n}{n^{2/3}} = 2^{-1/3}$. Thus, the expansion of $\frac{p_n^3}{27}+\frac{q_n^2}{4}$ contains terms of order $n^4$ (the boxed terms). Nevertheless, the overall order is strictly less than $n^4$, as the following calculation shows:

{\footnotesize\begin{align}
&\lim_{n\to\infty}\frac{1}{n^4}\left(-\frac{n^{6}}{27}X_{3}-\frac{n^{4}z_{n}^{2}}{9}X_{2}+\frac{n^{4}}{4}-\frac{n^{2}z_{n}^{4}}{9}X_{1}-\frac{z_{n}^{6}}{27}\right)=\nonumber  \\
&\lim_{n\to\infty}\frac{1}{108}\left(-4\left(\frac{n^{2/3}}{z_n}\right)^3\frac{z_n^3}{(z_n-1)^3}-12\frac{z_n^2}{(z_n-1)^2}+27-12\left(\frac{z_{n}}{n^{2/3}}\right)^3\frac{z_n}{z_n-1}-4\left(\frac{z_{n}}{n^{2/3}}\right)^6\right)=\nonumber\\&\lim_{n\to\infty}\frac{1}{108}\left(-4\cdot 2-12+27-12\cdot 2^{-1}-4\cdot 2^{-2}\right)=0.\nonumber
\end{align}}


\begin{table}[ht]
\centering
{\begin{tabular}{ |>{\bfseries}c|c|>{\bfseries}c|c|>{\bfseries}c|c|>{\bfseries}c|c|>{\bfseries}c|c|>{\bfseries}c|c| }
 \hline
   $\mathbf{n}$ & $m_n$ & $\mathbf{n}$ & $m_n$ & $\mathbf{n}$ & $m_n$ & $\mathbf{n}$ & $m_n$ & $\mathbf{n}$ & $m_n$ & $\mathbf{n}$ & $m_n$ \\
 \hline
1 & 1&36 & 174&71 & 449&106 & 782&141 & 1161&176 & 1576\\\hline
2 & 4&37 & 181&72 & 458&107 & 793&142 & 1174&177 & 1589\\\hline
3 & 5&38 & 188&73 & 467&108 & 802&143 & 1185&178 & 1602\\\hline
4 & 8&39 & 195&74 & 476&109 & 811&144 & 1196&179 & 1615\\\hline
5 & 11&40 & 202&75 & 485&110 & 822&145 & 1207&180 & 1628\\\hline
6 & 14&41 & 209&76 & 494&111 & 833&146 & 1218&181 & 1641\\\hline
7 & 17&42 & 216&77 & 503&112 & 844&147 & 1229&182 & 1654\\\hline
8 & 22&43 & 223&78 & 512&113 & 855&148 & 1240&183 & 1665\\\hline
9 & 25&44 & 230&79 & 521&114 & 866&149 & 1253&184 & 1678\\\hline
10 & 28&45 & 237&80 & 530&115 & 875&150 & 1266&185 & 1691\\\hline
11 & 33&46 & 244&81 & 539&116 & 886&151 & 1277&186 & 1704\\\hline
12 & 38&47 & 253&82 & 548&117 & 897&152 & 1288&187 & 1717\\\hline
13 & 41&48 & 260&83 & 557&118 & 908&153 & 1299&188 & 1730\\\hline
14 & 46&49 & 267&84 & 566&119 & 919&154 & 1310&189 & 1743\\\hline
15 & 51&50 & 274&85 & 575&120 & 930&155 & 1321&190 & 1756\\\hline
16 & 56&51 & 281&86 & 586&121 & 941&156 & 1334&191 & 1769\\\hline
17 & 61&52 & 290&87 & 595&122 & 952&157 & 1347&192 & 1782\\\hline
18 & 66&53 & 299&88 & 604&123 & 963&158 & 1360&193 & 1795\\\hline
19 & 71&54 & 306&89 & 613&124 & 974&159 & 1371&194 & 1808\\\hline
20 & 76&55 & 313&90 & 622&125 & 985&160 & 1382&195 & 1821\\\hline
21 & 81&56 & 320&91 & 631&126 & 996&161 & 1393&196 & 1834\\\hline
22 & 88&57 & 329&92 & 642&127 & 1007&162 & 1404&197 & 1847\\\hline
23 & 93&58 & 338&93 & 653&128 & 1018&163 & 1417&198 & 1860\\\hline
24 & 98&59 & 347&94 & 662&129 & 1029&164 & 1430&199 & 1873\\\hline
25 & 103&60 & 354&95 & 671&130 & 1040&165 & 1443&200 & 1886\\\hline
26 & 110&61 & 361&96 & 680&131 & 1051&166 & 1456&201 & 1899\\\hline
27 & 117&62 & 370&97 & 689&132 & 1062&167 & 1467&202 & 1912\\\hline
28 & 122&63 & 379&98 & 700&133 & 1073&168 & 1478&203 & 1925\\\hline
29 & 127&64 & 388&99 & 711&134 & 1084&169 & 1489&204 & 1938\\\hline
30 & 134&65 & 397&100 & 722&135 & 1095&170 & 1502&205 & 1951\\\hline
31 & 141&66 & 404&101 & 731&136 & 1106&171 & 1515&206 & 1964\\\hline
32 & 148&67 & 413&102 & 740&137 & 1117&172 & 1528&207 & 1977\\\hline
33 & 153&68 & 422&103 & 749&138 & 1128&173 & 1541&208 & 1990\\\hline
34 & 160&69 & 431&104 & 760&139 & 1139&174 & 1554&209 & 2003\\\hline
35 & 167&70 & 440&105 & 771&140 & 1150&175 & 1565&210 & 2016\\\hline
\end{tabular}}
\caption{The minimal sum of squares over partitions of $n$ with a nonnegative rank, for $1\leq n\leq 210$.}
\label{table:22}
\end{table}


\begin{thebibliography}{99}
        \bibitem{A}
        G.~E.~Andrews, 
        {\it The {T}heory of {P}artitions}, 
        Cambridge Univ. Press, 1984.

		\bibitem{DP}
		C.~Defant and J.~Propp,
		Quantifying noninvertibility in discrete dynamical systems,
		{\it Electron. J. Combin.}
		{\bf 27} (2020),
		Article P3.51.

		\bibitem{D}
		F.~J.~Dyson,
		Some guesses in the theory of partitions,
		{\it Eureka (Cambridge)}
		{\bf 8} (1944),
		10--15.
		
		\bibitem{G}
		R.~L.~Graham, D.~E.~Knuth, O.~Patashnik and S.~Liu,
        {\it Concrete {M}athematics}, Addison-Wesley, 1989.
        		
        \bibitem{V}
		E.~B.~Vinberg,
		{\it A {C}ourse in {A}lgebra},
		Amer. Math. Soc. Press, 2003.		
\end{thebibliography}
\end{document}